\documentclass[11pt,a4paper]{article}
\usepackage{amsfonts,amsgen,amstext,amsbsy,amsopn,amsfonts,amssymb,amscd}
\usepackage[leqno]{amsmath}
\usepackage[amsmath,amsthm,thmmarks]{ntheorem}
\usepackage{epsf,epsfig}
\usepackage{float}
\usepackage{dsfont}
\usepackage{ebezier,eepic}
\usepackage{color}
\usepackage{tikz}
\usepackage{multirow}
\usepackage{mathrsfs}
\usepackage{graphicx}
\usepackage{subfigure}
\newtheorem{thm}{Theorem}[section]

\newtheorem{prop}[thm]{Proposition}

\newtheorem{lem}[thm]{Lemma}
\newtheorem{example}[thm]{Example}
\newtheorem{false statement}{False statement}
\newtheorem{cor}[thm]{Corollary}

\newtheorem{fact}[thm]{Fact}

\theoremstyle{definition}
\newtheorem{defn}[thm]{Definition}
\newtheorem{claim}[thm]{Claim}

\newtheorem{conj}[thm]{Conjecture}

\makeatletter \@addtoreset{equation}{section}

\def\hl{\mathcal{L}}
\def\hht{\mathcal{T}}

\def\hf{\mathcal{F}}
\def\hg{\mathcal{G}}

\def\ha{\mathcal{A}}
\def\hb{\mathcal{B}}

\begin{document}
\title{\bf\Large On the $C$-diversity of intersecting hypergraphs}
\date{}
\author{Peter Frankl$^1$, Jian Wang$^2$\\[10pt]
$^{1}$R\'{e}nyi Institute, Budapest, Hungary\\[6pt]
$^{2}$Department of Mathematics\\
Taiyuan University of Technology\\
Taiyuan 030024, P. R. China\\[6pt]
E-mail:  $^1$frankl.peter@renyi.hu, $^2$wangjian01@tyut.edu.cn
}

\maketitle

\begin{abstract}
Let $\mathcal{F}\subset \binom{X}{k}$ be a family consisting of $k$-subsets of the $n$-set $X$. Suppose that $\mathcal{F}$ is intersecting, i.e., $F\cap F'\neq \emptyset$ for all $F,F'\in \mathcal{F}$. Let $\Delta(\mathcal{F})$ be the maximum degree of $\mathcal{F}$. For a constant $C\geq 1$ the {\it $C$-diversity} (cf. \cite{MPW}), $\gamma_C(\mathcal{F})$ is defined as $|\mathcal{F}|-C\Delta(\mathcal{F})$. Define $\mathcal{F}_{123} =\left\{F\in \binom{X}{k}\colon |F\cap \{1,2,3\}|=2\right\}$. It has $C$-diversity $(3-2C)\binom{n-3}{k-2}$. The main result shows that for $1< C<\frac{3}{2}$ and $n\geq \frac{42}{3-2C}k$, $\gamma_C(\mathcal{F})\leq \gamma_C(\mathcal{F}_{123})$ with equality if and only if $\mathcal{F}$ is isomorphic to $\mathcal{F}_{123}$. For the case of ordinary diversity $(C=1)$ a strong stability is proven.
\end{abstract}

\section{Introduction}

Let $[n]=\{1,2,\ldots,n\}$ be the standard $n$-element set. Let $\binom{[n]}{k}$ denote the collection of all $k$-subsets of $[n]$. A subset $\hf\subset \binom{[n]}{k}$ is called a {\it $k$-uniform family} or simply {\it $k$-graph}. For $x\in [n]$, let
\[
\hf(x)=\left\{F\setminus \{x\}\colon x\in F\in \hf\right\},\ \hf(\bar{x})= \{F\colon x\notin F\in \hf\}
\]
and note that $|\hf|=|\hf(x)|+|\hf(\bar{x})|$. Define $\Delta(\hf)=\max_{1\leq x\leq n} |\hf(x)|$, {\it the maximum degree}.  Define
\[
\varrho(\hf) = \frac{\Delta(\hf)}{|\hf|},\ \gamma(\hf) =|\hf|-\Delta(\hf),
\]
where $\gamma(\hf)$ is called the {\it diversity} of $\hf$.

A family $\hf\subset \binom{[n]}{k}$ is called {\it intersecting} if $F\cap F'\neq \emptyset$ for all $F,F'\in \hf$. Let us recall the Erd\H{o}s-Ko-Rado Theorem.

\begin{thm}[\cite{ekr}]
Let $n\geq 2k>0$. Suppose that $\hf\subset \binom{[n]}{k}$ is intersecting. Then
\begin{align}\label{ineq-ekr}
|\hf| \leq \binom{n-1}{k-1}.
\end{align}
\end{thm}

The {\it full star} $\hht(n,k)=\left\{T\in \binom{[n]}{k}\colon 1\in T\right\}$ shows that \eqref{ineq-ekr} is best possible.

Let us define some more intersecting families: For $3\leq i\leq k+1$,
\[
\hf_i=\hf_i(n,k)=\left\{F\in \binom{[n]}{k}\colon 1\in F,\ F\cap[2,i]\neq \emptyset\right\}\bigcup \left\{F\in \binom{[n]}{k}\colon [2,i]\subset F \right\}.
\]
Note that
\[
|\hf_i| =\binom{n-1}{k-1} -\binom{n-i}{k-1}+\binom{n-i}{k-i+1}\mbox{ and } |\hf_i(\bar{1})| =\binom{n-i}{k-i+1}.
\]
It is not hard to see that for $n>2k$,
\begin{align}\label{ineq-fisequence}
|\hf_3|=|\hf_4|<|\hf_5|<\ldots<|\hf_{k+1}|<\binom{n-1}{k-1}.
\end{align}

Hilton and Milner proved the following stability result.

\begin{thm}[\cite{HM}]\label{thm-hm}
Let $n>2k\geq 4$. Suppose that $\hf\subset \binom{[n]}{k}$ is intersecting but $\hf$ is not a star. Then
\begin{align}\label{ineq-hm}
|\hf| \leq |\hf_{k+1}|=\binom{n-1}{k-1}-\binom{n-k-1}{k-1}+1.
\end{align}
\end{thm}

In \cite{F87-2} the following generalisation of \eqref{ineq-hm} was proved.

\begin{thm}\label{thm-frankl}
Let $n>2k\geq 4$. Fix an integer $i$, $3\leq i\leq k+1$. Suppose that $\hf\subset \binom{[n]}{k}$ is intersecting  and $\Delta(\hf)\leq \Delta(\hf_i)$. Then
\begin{align}\label{ineq-frankl}
|\hf|\leq |\hf_i|.
\end{align}
\end{thm}

Note that if $\hf\subset \binom{[n]}{k}$ is intersecting and $x\in [n]$, $F_0\in \hf$ satisfy $x\notin F_0$ then $G\cap F_0\neq \emptyset$ for all $G\in \hf(x)$. Hence $|\hf(x)| \leq \binom{n-1}{k-1}-\binom{n-k-1}{k-1}=\Delta(\hf_{k+1})$. Consequently, Theorem \ref{thm-frankl} implies Theorem \ref{thm-hm}.

As pointed out by Kupavskii and Zakharov \cite{KZ}, the actual proof in \cite{F87-2} yields the corresponding result for diversity as well.

\begin{thm}
Let $n>2k\geq 4$, $3\leq i\leq k+1$. Suppose that $\hf\subset \binom{[n]}{k}$ is intersecting  and $\gamma(\hf)=|\hf|-\Delta(\hf)\geq \gamma(\hf_i)$ then
\begin{align}
|\hf| \leq |\hf_i|.
\end{align}
\end{thm}

\begin{defn}[Pure triangle family]
For $n>2k\geq 4$ define
\[
\hf_{\bigtriangleup} =\left\{F\in\binom{[n]}{k}\colon |F\cap [3]|=2\right\}.
\]
\end{defn}

Lemmons and Palmer \cite{LP} proved that for $n\geq 6k^3$ the diversity of an intersecting family never exceeds $\binom{n-3}{k-2}=\gamma(\hf_3)$. This was subsequently improved by \cite{F17}, \cite{Ku1} and \cite{F2020}. The current record is the following.

\begin{thm}[\cite{FW2022}]\label{thm-fw-2}
Let $n>36k$. Suppose that $\hf\subset \binom{[n]}{k}$ is intersecting. Then
\begin{align}
\gamma(\hf) \leq \binom{n-3}{k-2}
\end{align}
with equality holding if and only if $\hf_{\bigtriangleup}\subset\hf\subset\hf_3$ up to isomorphism.
\end{thm}

Let us  introduce  the central notion of the present paper, which is due to Magnan, Palmer and Wood \cite{MPW}.

\begin{defn}
Let $C>0$ be an absolute constant and let $\hf\subset \binom{[n]}{k}$ be an intersecting family. Define the $C$-diversity $\gamma_C(\hf)$ by
\[
\gamma_C(\hf)=|\hf|-C\Delta(\hf).
\]
\end{defn}

If $C\geq 3/2$ then $\gamma_C(\hf_3)<0$. Also if $C\leq 1$ then $\gamma_C(\hf) \geq \gamma(\hf)$ and $\gamma_C(\hf)\geq \gamma_C(\hg)$ whenever $\hf \supset \hg$. These observations should explain that $1<C<3/2$ is range that $\gamma_C(\hf)$ provides meaningful information on $\hf$ as well as estimating its value is more challenging.

Note that $\hf_{\bigtriangleup}\subset \hf_3$, $|\hf_{\bigtriangleup}|=3\binom{n-3}{k-2}$, $\gamma_C(\hf_{\bigtriangleup})=(3-2C)\binom{n-3}{k-2}$. Moreover, if $\hf_{\bigtriangleup}\subsetneq \hf\subset \hf_3$ then $\gamma(\hf_{\bigtriangleup})=\gamma(\hf)$ but $\gamma_C(\hf)<\gamma_C(\hf_{\bigtriangleup})$.

Recall that an intersecting family $\hf\subset \binom{[n]}{k}$ is called {\it saturated} if $\hf\cup\{G\}$ ceases to be intersecting for each $G\in \binom{[n]}{k}\setminus \hf$.

Magnan, Palmer and Wood \cite{MPW} show that for given $C$, $1<C<3/2$ and $n>n_0(k)$ the maximal $C$-diversity is attained by $\hf_{\bigtriangleup}$. Since saturatedness cannot be assumed, they rely on the sunflower method from \cite{F78-2}. Consequently, they need $n$ to be exponentially large with respect to $k$.

We use different methods and show that the same statement holds under linear constraints.

\begin{thm}\label{thm-main}
Let $1< C< \frac{3}{2}$.  Suppose that $\hf\subset \binom{[n]}{k}$ is an intersecting family with $k\geq 3$ and $n\geq \frac{42k}{3-2C}$. Then
\[
\gamma_C(\hf) \leq (3-2C)\binom{n-3}{k-2}
\]
with equality holding if and only if $\hf=\hf_{\bigtriangleup}$ up to isomorphism.
\end{thm}

For $\{u,v,w\}\subset [n]$, let us introduce the notation:
\[
\hf_{uvw} =\left\{F\in \binom{[n]}{k}\colon |F\cap T|=2\right\}, \ \hf_{uvw}^* =\left\{F\in \binom{[n]}{k}\colon |F\cap T|\geq 2\right\}.
\]

In \cite{MPW}, Magnan, Palmer and Wood proved a stability result for Theorem \ref{thm-fw-2}.

\begin{thm}[\cite{MPW}]
Let $k\geq 3$ and $t$ a positive integer. If $\hf\subset \binom{[n]}{k}$ is an intersecting family with diversity $\gamma(\hf) =\binom{n-3}{k-2}-t$ and $n$ is sufficiently large, then there exists $\{u,v,w\}\subset [n]$ such that $|\hf_{uvw}\setminus \hf| \leq 3t$.
\end{thm}

In the present paper, we improve their result as follows:

\begin{thm}\label{thm-main2}
Let $\hf\subset \binom{[n]}{k}$ be an intersecting family. Define $\alpha$ by $\gamma(\hf)=\left(1-\alpha\right)\binom{n-3}{k-2}$. Let $d$ be an integer satisfying $d\geq 36$. If $0\leq  \alpha<1$ and  $n\geq \frac{dk}{1-\alpha}$, then there exists $\{u,v,w\}\subset [n]$ such that
\begin{align}
&|\hf\setminus \hf_{uvw}^*| \leq  \frac{d\alpha}{2} \binom{n-d+3}{k-d+3}\mbox{ and }\label{ineq4-1}\\[5pt]
&|\hf_{uvw}\setminus \hf|\leq 3\alpha\binom{n-3}{k-2}+\frac{3d\alpha}{2} \binom{n-d+3}{k-d+3}.\label{ineq4-2}
\end{align}
\end{thm}

\begin{example}
Let $A_i\subset [4,n]$, $2\leq |A_i|<k$ and suppose that $A_i\cap A_j\neq \emptyset$, $1\leq i,j\leq 3$. Define $\hht=\hht_1\cup \hht_2\cup \hht_3$ where
\[
\hht_i =\left\{F\in \binom{[n]}{k}\colon F\cap [3]=[3]\setminus \{i\},\ F\cap A_i\neq \emptyset\right\}\cup \left\{F\in \binom{[n]}{k}\colon F\cap [3]=\{i\},\ A_i\subset F\right\}.
\]
It is easy to see that $\hf$ is intersecting. If $|A_1|=|A_2|=|A_3|=:\ell$ then
\[
\gamma(\hht) =\binom{n-3}{k-2}-\binom{n-3-\ell}{k-2}+2\binom{n-3-\ell}{k-\ell-1}
\]
and
\[
|\hht\setminus \hf_{123}^*| = 3\binom{n-3-\ell}{k-\ell-1},\ | \hf_{123}\setminus\hht| = 3\binom{n-3-\ell}{k-2}.
\]
It shows that \eqref{ineq4-1} and \eqref{ineq4-2} are not too far from best possible.
\end{example}

Assuming $n\geq \frac{k(k+4)}{1-\alpha}$ we can obtain a much stronger stability result.

\begin{cor}\label{cor-1}
Let $\hf\subset \binom{[n]}{k}$ be an intersecting family. If $\gamma(\hf)=\left(1-\alpha\right)\binom{n-3}{k-2}$ with $0\leq  \alpha<1$ and  $n\geq \max\left\{k+4,40\right\}\cdot \frac{k}{1-\alpha}$, then there exist $u,v,w\in [n]$ such that  $\hf\subset \hf_{uvw}^*$ and
\[
|\hf_{uvw}\setminus \hf| \leq 3\alpha\binom{n-3}{k-2}.
\]
\end{cor}
\begin{proof}
Set $d=\max\left\{k+4,40\right\}$. Then $n\geq \frac{dk}{1-\alpha}$ and $d\geq 40$. Note that $d\geq k+4$ implies $\binom{n-d+3}{k-d+3}=0$. Thus the corollary follows from Theorem \ref{thm-main2}.
\end{proof}

\begin{cor}
Let $\hf\subset \binom{[n]}{k}$ be an intersecting family with $n> 36k$. If $|\hf|\geq |\hf_3|+1$, then
\[
\varrho(\hf) > 1- \max\left\{k+4,40\right\}\cdot \frac{k}{3n}.
\]
If $|\hf|\geq |\hf_3|+\frac{d}{2} \binom{n-d+3}{k-d+3}$ for some $d\geq 40$, then
\[
\varrho(\hf) > 1-  \frac{dk}{3n}.
\]
\end{cor}

\begin{proof}
Define $\alpha$ by $\gamma(\hf)=\left(1-\alpha\right)\binom{n-3}{k-2}$. By Theorem \ref{thm-fw-2} we infer that $0\leq \alpha\leq 1$.
If $\alpha=1$ then $\varrho(\hf) =1$. Thus we may assume $0\leq \alpha<1$.

Note that $|\hf|\geq |\hf_3|+1$ implies that   $\hf\not\subset \hf_{uvw}^*$ for any $u,v,w\in [n]$.   Then by Corollary \ref{cor-1} it is only possible if
\[
n<\max\left\{k+4,40\right\}\cdot \frac{k}{1-\alpha}.
\]
It follows that
\[
1-\alpha <\max\left\{k+4,40\right\}\cdot \frac{k}{n}.
\]
Thus,
\[
\varrho(\hf) =1-\frac{\gamma(\hf)}{|\hf|}\geq 1-\frac{\left(1-\alpha\right)\binom{n-3}{k-2}}{|\hf_3|} > 1-\frac{1-\alpha}{3} > 1- \max\left\{k+4,40\right\}\cdot \frac{k}{3n}.
\]

If $|\hf|\geq |\hf_3|+\frac{d}{2} \binom{n-d+3}{k-d+3}$ for some $d\geq 40$, then for any $u,v,w\in [n]$ and $\alpha<1$ we have
\[
|\hf\setminus \hf_{uvw}^*| >  \frac{d\alpha}{2} \binom{n-d+3}{k-d+3}.
\]
By Theorem \ref{thm-main2}, it is only possible when $n< \frac{dk}{1-\alpha}$. It follows that $\alpha >1-\frac{dk}{n}$. Thus,
\[
\varrho(\hf) =1-\frac{\gamma(\hf)}{|\hf|}\geq 1-\frac{\left(1-\alpha\right)\binom{n-3}{k-2}}{|\hf_3|} > 1-\frac{1-\alpha}{3} > 1-  \frac{dk}{3n}.
\]

\end{proof}

The proofs are heavily relying on the following recent result concerning $\varrho(\hf)$.

\begin{thm}[\cite{FW2022}]\label{thm-fw-3}
Suppose that $\hf\subset \binom{[n]}{k}$ is an intersecting family with $|\hf|\geq 36\binom{n-3}{k-3}$ and  $n\geq 24k$. Then \begin{align}\label{ineq-2}
\varrho(\hf)>\frac{2}{3}-\frac{k}{n}.
\end{align}
\end{thm}

The paper is organized as follows. In Section 2, we summarize some further tools for the proof. In Section 3 we prove Theorem \ref{thm-main}. In Section 4 we prove Theorem  \ref{thm-main2}. In Section 5 we present some related open problems.

\section{Some further tools}

%
In this section we recall some results that are needed in the proof.

For $\hf\subset \binom{[n]}{k}$ and $0\leq \ell<k$, define
\[
\partial^{(\ell)}\hf =\left\{E\in \binom{[n]}{\ell}\colon \mbox{ there exists }F\in \hf \mbox{ such that }E\subset F\right\}.
\]

\begin{prop}[\cite{Sperner}]\label{sperner}
For $\hf\subset \binom{[n]}{k}$ and $0\leq \ell\leq k$,
\begin{align}\label{ineq-sperner}
\frac{|\partial^{(\ell)}\hf|}{\binom{n}{\ell}} \geq \frac{|\hf|}{\binom{n}{k}}.
\end{align}
\end{prop}

 We  need the  following form of Sperner's bound (Proposition \ref{sperner}).

 \begin{cor}\label{cor-key}
 Let $n\geq a+b$ and let $\ha\subset \binom{[n]}{a}$, $\hb\subset \binom{[n]}{b}$ be cross-intersecting families. Then
 \begin{align}\label{ineq-sperner2}
 \frac{|\ha|}{\binom{n}{a}}+ \frac{|\hb|}{\binom{n}{b}}\leq 1.
 \end{align}
 \end{cor}
 \begin{proof}
 Let $\ha^{c}=\{[n]\setminus A\colon A\in \ha\}$. Since $\ha,\hb$ are cross-intersecting, $\partial^{(b)} \ha^{c}\cap \hb=\emptyset$. It follows that
 \[
 |\partial^{(b)} \ha^{c}|+|\hb|\leq \binom{n}{b}.
 \]
 Note that $n\geq a+b$ implies $n-a\geq b$. By Sperner's bound (Proposition \ref{sperner}), we have
 \[
 \frac{|\partial^{(b)} \ha^{c}|}{\binom{n}{b}} \geq \frac{|\ha^{c}|}{\binom{n}{n-a}}=  \frac{|\ha|}{\binom{n}{a}}.
 \]
 Thus \eqref{ineq-sperner2} follows.
 \end{proof}

Recall the {\it lexicographic order} $A <_{L} B$ for $A,B\in \binom{[n]}{k}$ defined by, $A<_L B$ iff $\min\{i\colon i\in A\setminus B\}<\min\{i\colon i\in B\setminus A\}$. For $n>k>0$ and $\binom{n}{k}\geq m>0$ let $\hl(n,k,m)$ denote the first $m$ sets $A\in \binom{[n]}{k}$ in the lexicographic order.

Two families $\ha\subset \binom{[n]}{a}$ and $\hb\subset \binom{[n]}{b}$ are called {\it cross-intersecting} if $A\cap B \neq \emptyset$ for all $A\in \ha$, $B\in \hb$.

{\noindent\bf Hilton's Lemma (\cite{Hilton}).} Let $n,a,b$ be positive integers, $n\geq a+b$. Suppose that $\ha\subset \binom{[n]}{a}$ and $\hb\subset \binom{[n]}{b}$ are cross-intersecting. Then $\hl(n,a,|\ha|)$ and $\hl(n,b,|\hb|)$ are cross-intersecting as well.

The following  lemma is an easy consequence of Hilton's Lemma.

\begin{lem}\label{lem-key0}
Let $\ha\subset\binom{[n]}{a}$, $\hb\subset \binom{[n]}{b}$ be cross-intersecting families with $n\geq a+b$. If $|\ha|\geq \binom{n-1}{a-1}+\binom{n-2}{a-1}+\ldots+\binom{n-d}{a-1}$, $d<b$, then $|\hb| \leq \binom{n-d}{b-d}$.
\end{lem}
\begin{proof}
By Hilton's Lemma, we may assume that $\ha=\hl(n,a,|\ha|)$ and $\hb=\hl(n,b,|\hb|)$. Since $|\ha|\geq \binom{n-1}{a-1}+\binom{n-2}{a-1}+\ldots+\binom{n-d}{a-1}$,
\[
\left\{A\in \binom{[n]}{a}\colon A\cap [d]\neq\emptyset\right\} \subset \ha.
\]
It follows that $[d]\subset B$ for all $B\in \hb$. Thus $|\hb| \leq \binom{n-d}{b-d}$.
\end{proof}

Combining Hilton's Lemma and Sperner's bound, one can obtain a useful inequality for cross-intersecting families.

\begin{lem}\label{lem-key}
Let $\ha\subset\binom{[n]}{a}$, $\hb\subset \binom{[n]}{b}$ be cross-intersecting families with $n\geq a+b$. If $|\ha|\geq \binom{n-1}{a-1}+\binom{n-2}{a-1}+\ldots+\binom{n-d}{a-1}$, $d<b$, then
\[
|\ha|+\frac{\binom{n-d}{a}}{\binom{n-d}{b-d}}|\hb| \leq \binom{n}{a}.
\]
\end{lem}

\begin{proof}
By Hilton's Lemma, assume that $\ha=\hl(n,a,|\ha|)$ and $\hb=\hl(n,b,|\hb|)$. Since $|\ha|\geq \binom{n-1}{a-1}+\binom{n-2}{a-1}+\ldots+\binom{n-d}{a-1}$,
\[
\left\{A\in \binom{[n]}{a}\colon A\cap [d]\neq\emptyset\right\} \subset \ha.
\]
It follows that
\[
|\ha|= \binom{n-1}{a-1}+\binom{n-2}{a-1}+\ldots+\binom{n-d}{a-1}+|\ha(\overline{[d]})| \mbox{ and }|\hb|=|\hb([d])|.
\]
 Now $\ha(\overline{[d]})$  and $\hb([d])$ are cross-intersecting. Applying \eqref{ineq-sperner2} we obtain that
\[
 \frac{|\ha(\overline{[d]})|}{\binom{n-d}{a}}+ \frac{|\hb([d])|}{\binom{n-d}{b-d}}\leq 1.
\]
Then
\[
|\ha(\overline{[d]})|+\frac{\binom{n-d}{a}}{\binom{n-d}{b-d}}|\hb|\leq \binom{n-d}{a}.
\]
Adding $\binom{n-1}{a-1}+\binom{n-2}{a-1}+\ldots+\binom{n-d}{a-1}$ on both sides, we get
\[
|\ha|+\frac{\binom{n-d}{a}}{\binom{n-d}{b-d}}|\hb| \leq \binom{n}{a}.
\]
\end{proof}

 We need the following result concerning the diversity of cross-intersecting families. In \cite{F2020}, it is stated that Lemma \ref{thm-fk} below holds under the condition $|\ha|,|\hb|\geq \frac{1}{2}\binom{m-1}{\ell-1}$ and $(m-1)>10(\ell-1)$. Actually, the same proof works also for $|\ha|,|\hb|\geq 5\binom{m-2}{\ell-2}$. For self-containedness, we repeat the proof in the Appendix.

\begin{lem}[\cite{F2020}]\label{thm-fk}
Let $m,\ell$ be integers, $m\geq 2\ell$, $\ell\geq 2$. Suppose that $\ha,\hb\subset \binom{[m]}{\ell}$ are cross-intersecting families with $|\ha|,|\hb|\geq 5\binom{m-2}{\ell-2}$. Then there exists $j\in [m]$ such that
\begin{align}\label{ineq-frankl2}
\max\left\{|\ha(\bar{j})|,|\hb(\bar{j})|\right\} \leq \binom{m-2}{\ell-2}.
\end{align}
\end{lem}

Let us prove a simple but useful statement.

\begin{fact}\label{fact-key}
Let $\hf\subset \binom{[n]}{k}$ and  $u,v\in [n]$.
Suppose that $|\hf(u)|\geq |\hf(v)|$ then $|\hf(u,\bar{v})|\geq |\hf(\bar{u},v)|$.
\end{fact}

\begin{proof}
Just note that $|\hf(x)|=|\hf(x,y)|+|\hf(x,\bar{y})|$ implies $|\hf(u)|-|\hf(v)| = |\hf(u,\bar{v})|-|\hf(\bar{u},v)|$.
\end{proof}

For $\hf\subset \binom{[n]}{k}$, the {\it matching number} $\nu(\hf)$ is defined as  the maximum number of pairwise disjoint members in $\hf$. Recall another simple statement.

\begin{prop}[\cite{F87}]
Suppose that $\hf\subset \binom{[n]}{k}$ then
\begin{align}\label{ineq-EMCup}
|\hf|\leq \nu(\hf)\binom{n-1}{k-1}.
\end{align}
\end{prop}

We also need the following inequality.

\begin{prop}[\cite{FW2022-2}]
  Let $n,k,i$ be positive integers. Then
  \begin{align}\label{ineq4-1.2}
    \binom{n-i}{k}\geq \frac{n-ik}{n}\binom{n}{k}, \mbox{ for } n>ik.
  \end{align}
\end{prop}

For $\hf\subset\binom{[n]}{k}$ and $P\subset Q\subset [n]$, let
\[
\hf(P,Q) = \left\{F\setminus Q\colon F\cap Q=P, F\in\hf\right\}.
\]
We also use $\hf(x,\bar{y},\bar{z})$ to denote $\hf(\{x\},\{x,y,z\}$  and  $\hf(x,y,\bar{z})$ to denote $\hf(\{x,y\}, \{x,y,z\})$.

\section{Proof of Theorem \ref{thm-main}}

Let us first prove a useful lemma.

\begin{lem}\label{lem-key2}
Let $\hf\subset \binom{[n]}{k}$ be an intersecting family  and let $|\hf(u)|=\Delta(\hf)$. If  $|\hf(\bar{u},v)|\geq  5\binom{n-4}{k-3}$ for some $v\in [n]\setminus \{u\}$, then there exists $w\in [n]\setminus \{u,v\}$ such that
\begin{align*}
 |\hf(\emptyset,\{u,v,w\})|\leq \binom{n-7}{k-4}\mbox{ and }|\hf(\{x\},\{u,v,w\})|\leq \binom{n-4}{k-3},\ x=u,v,w.
\end{align*}
\end{lem}
\begin{proof}
Since $|\hf(u)|=\Delta(\hf)\geq |\hf(v)|$, by Fact \ref{fact-key} we have
\[|\hf(u,\bar{v})|\geq |\hf(\bar{u},v)|\geq 5\binom{n-4}{k-3}.
\]
Note that $\hf(u,\bar{v}),\hf(\bar{u},v)\subset \binom{[n]\setminus\{u,v\}} {k-1}$ are cross-intersecting. Applying Lemma \ref{thm-fk} with $m=n-2$ and $\ell=k-1$, there exists $w\in [n]\setminus\{u,v\}$ such that
\begin{align}\label{ineq-key4}
|\hf(u,\bar{v},\bar{w})|,|\hf(\bar{u},v,\bar{w})|\leq \binom{n-4}{k-3}.
\end{align}
Let $T=\{u,v,w\}$. Then
\[
|\hf(\{u,w\},T)|= |\hf(u,\bar{v})|-|\hf(u,\bar{v},\bar{w})| \geq 5\binom{n-4}{k-3}-\binom{n-4}{k-3}=4\binom{n-4}{k-3}.
\]
Note that $\hf(\{u,w\},T)$ and $\hf(\emptyset,T)$ are cross-intersecting. By Lemma \ref{lem-key0} we have
\begin{align}\label{ineq-key7}
|\hf(\emptyset,T)|\leq \binom{n-7}{k-4}.
\end{align}

Note that
\[
|\hf(\bar{w})| \geq |\hf(\bar{u})|\geq |\hf(\bar{u},v)| \geq 5\binom{n-4}{k-3}.
\]
Thus,
\begin{align*}
|\hf(\{u,v\},T)|&= |\hf(\bar{w})| - |\hf(\{u\},T)|- |\hf(\{v\},T)|-|\hf(\emptyset,T)|\\[5pt]
&\geq 5\binom{n-4}{k-3}-2\binom{n-4}{k-3}-\binom{n-7}{k-4}\\[5pt]
&> \binom{n-4}{k-3}+\binom{n-5}{k-3}+\binom{n-6}{k-3}.
\end{align*}
Since $\hf(\{u,v\},T), \hf(\{w\},T)$ are cross-intersecting, by Lemma \ref{lem-key0}
\begin{align}\label{ineq-key6}
|\hf(\{w\},T)|\leq \binom{n-6}{k-4}<\binom{n-4}{k-3}.
\end{align}
Combining \eqref{ineq-key4}, \eqref{ineq-key7} and \eqref{ineq-key6} we see that the lemma holds.
\end{proof}

Now we are in a position to prove the main result.

\begin{proof}[Proof of Theorem \ref{thm-main}]
Let us assume indirectly that
\begin{align}\label{ineq-assumption}
\gamma_C(\hf)>(3-2C)\binom{n-3}{k-2} \mbox{ or }\gamma_C(\hf)=(3-2C)\binom{n-3}{k-2}\mbox{ but $\hf$ is not isomorphic to } \hf_{\triangle}
\end{align}
 and derive a contradiction.

%
%

\begin{claim}\label{claim-1}
$\frac{8}{3}\binom{n-3}{k-2}< |\hf|< 3\binom{n-3}{k-2}$.
\end{claim}
\begin{proof}
If $|\hf|\geq 3\binom{n-3}{k-2}$, then by Theorem \ref{thm-fw-2} and $n> 36k$,
\[
\gamma_C(\hf) =C\gamma(\hf) -(C-1)|\hf| \leq C\binom{n-3}{k-2}-(C-1)\cdot3\binom{n-3}{k-2}=(3-2C)\binom{n-3}{k-2}.
\]
The equality holds if and only if $\gamma (\hf)=\binom{n-3}{k-2}$ and $|\hf|=3\binom{n-3}{k-2}$. By Theorem \ref{thm-fw-2} we infer that the equality holds if and only if $\hf=\hf_{\bigtriangleup}$ up to isomorphism. This contradicts our assumption \eqref{ineq-assumption}.
Thus $|\hf|< 3\binom{n-3}{k-2}$.

Note that $3-2C<1$ implies $n-k\geq \frac{42k}{3-2C}-k >\frac{41k}{3-2C}$.
Thus,
\[
\binom{n-3}{k-2} =\frac{n-k}{k-2}\binom{n-3}{k-3}>\frac{41}{3-2C} \binom{n-3}{k-3}.
\]
Consequently,
\[
|\hf|>\gamma_C(\hf) \geq (3-2C) \binom{n-3}{k-2}>41\binom{n-3}{k-3}.
\]
 Set $\varepsilon=\frac{3}{2}-C$. Then $0<\varepsilon<\frac{1}{2}$ and $n\geq \frac{42k}{3-2C}=\frac{21k}{\varepsilon}$.  By Theorem \ref{thm-fw-3} we have $\varrho(\hf)>\frac{2}{3}-\frac{\varepsilon}{21}$,  that is,
\begin{align}\label{ineq-1}
\Delta(\hf) \geq \left(\frac{2}{3}-\frac{\varepsilon}{21}\right)|\hf|.
\end{align}

 If $|\hf| \leq  \frac{8}{3}\binom{n-3}{k-2}$ then by \eqref{ineq-1} we have
\begin{align*}
\gamma_C(\hf) = |\hf| -C \Delta(\hf) \leq \left(1-\frac{2C}{3}+\frac{C\varepsilon}{21}\right)|\hf|&\leq  \left(1-\frac{2C}{3}+\frac{C\varepsilon}{21} \right)\left(3-\frac{1}{3}\right)\binom{n-3}{k-2}.
\end{align*}
Let us show that the coefficient of $\binom{n-3}{k-2}$ on the right hand side is less than $3-2C$.
\begin{align*}
\left(1-\frac{2C}{3}+\frac{C\varepsilon}{21} \right)\left(3-\frac{1}{3}\right)&= (3-2C)+\frac{C\varepsilon}{7} - \frac{1}{3}\left(1-\frac{2C}{3}+\frac{C\varepsilon}{21}\right)\\[5pt]
&= (3-2C) - \frac{1}{3}\left(1-\frac{2C}{3}-\frac{8C\varepsilon}{21}\right).
\end{align*}
Using  $C= \frac{3}{2}-\varepsilon$ we infer
\[
1-\frac{2 C}{3} -\frac{8C\varepsilon}{21}> \frac{2}{3}\varepsilon -\frac{8}{21} \times \frac{3}{2} \varepsilon = \frac{2\varepsilon}{21}>0.
\]
It follows that $|\hf|<  (3-2C)\binom{n-3}{k-2}$, a contradiction. Thus $|\hf| > \frac{8}{3}\binom{n-3}{k-2}$.
\end{proof}

\begin{cor}\label{cor-3.3}
$\varrho(\hf)<\frac{2}{3}$.
\end{cor}
\begin{proof}
Define $\delta=\frac{|\hf|}{3\binom{n-3}{k-2}}$ and note $\delta<1$. If $\varrho(\hf)\geq \frac{2}{3}$ then
\[
\gamma_C(\hf) =|\hf| - C\varrho(\hf)|\hf|\leq \left(1-\frac{2C}{3}\right)|\hf| =\delta(3-2C) \binom{n-3}{k-2},
\]
contradicting \eqref{ineq-assumption}.
\end{proof}

Now we show that $\hf$ is nearly contained in $\hf_3$.

\begin{claim}
There exists $T=\{u,v,w\}\subset [n]$ such that
\begin{align}\label{ineq-key0}
 |\hf(\emptyset,T)|\leq \binom{n-7}{k-4}\mbox{ and }|\hf(\{x\},T)|\leq \binom{n-4}{k-3},\ x=u,v,w.
\end{align}
\end{claim}
\begin{proof}
Let $|\hf(u)|=\Delta(\hf)$ for some $u\in [n]$. By Claim \ref{claim-1} and Corollary \ref{cor-3.3},
\[
|\hf(\bar{u})|=|\hf|-\Delta(\hf)>\frac{|\hf|}{3} \geq \frac{8}{9} \binom{n-3}{k-2}.
\]
Since $n>42k$, we have
\[
|\hf(\bar{u})|\geq \frac{8}{9} \frac{n-3}{k-2}\binom{n-4}{k-3}\geq \frac{8}{9} \times 42 \binom{n-4}{k-3}> 36\binom{n-4}{k-3}.
\]
Since $\hf(\bar{u})$ is intersecting and $n-1\geq \frac{21k}{\varepsilon}-1> \frac{20k}{\varepsilon}$. By Theorem \ref{thm-fw-3}, there exists $v\in [n]\setminus \{u\}$ such that
\[
|\hf(\bar{u},v)| \geq \left(\frac{2}{3}-\frac{\varepsilon}{20}\right)|\hf(\bar{u})|\geq \left(\frac{2}{3}-\frac{\varepsilon}{20}\right) \times 36\binom{n-4}{k-3}
>5\binom{n-4}{k-3}.
\]
Then \eqref{ineq-key0} follows from Lemma \ref{lem-key2}.
\end{proof}

Now \eqref{ineq-key0}  shows $|\hf\setminus \hf_{uvw}^*|< 3\binom{n-3}{k-3}$ and $|\hf\setminus \hf_{uvw}|< 4\binom{n-3}{k-3}$.
That is, $\hf$ is quite close to the pure triangle family. We can proceed even further.

\begin{claim}
For $P\in \{\{u,v\}, \{u,w\}, \{v,w\}\}$,
\begin{align}\label{ineq-key5}
|\hf(P,T)|  > \frac{7}{9}\binom{n-3}{k-2}.
\end{align}
\end{claim}
\begin{proof}
Renaming $u$, $v$ and $w$ if necessary, assume that
\begin{align}\label{ineq-key2}
|\hf(\{u,v\},T)|\geq |\hf(\{u,w\},T)|\geq |\hf(\{v,w\},T)|.
\end{align}
If  $|\hf(\{u,v\},T)|+ |\hf(\{u,w\},T)|\geq 2|\hf(\{v,w\},T)|+4\binom{n-3}{k-3}$, then by \eqref{ineq-key0}  we have
\begin{align*}
|\hf(u)|&\geq |\hf(\{u,v\},T)|+ |\hf(\{u,w\},T)|\\[5pt]
&\geq 2\left(|\hf(\{v,w\},T)|+2\binom{n-3}{k-3}\right)\\[5pt]
&\geq 2\left(|\hf(\{v,w\},T)|+|\hf(\{v\},T)|+|\hf(\{w\},T)|+|\hf(\emptyset,T)|\right)\\[5pt]
&=2|\hf(\bar{u})|.
\end{align*}
It follows that $\Delta(\hf)\geq |\hf(u)|\geq \frac{2}{3}|\hf|$, contradicting Corollary \ref{cor-3.3}. Thus,
\[
|\hf(\{u,v\},T)|+ |\hf(\{u,w\},T)|< 2|\hf(\{v,w\},T)|+4\binom{n-3}{k-3}.
\]
Then
\[
|\hf(\{u,v\},T)|+|\hf(\{u,w\},T)|+|\hf(\{v,w\},T)|< 3|\hf(\{v,w\},T)|+4\binom{n-3}{k-3}.
\]
Note that
\[
|\hf(\{u,v\},T)|+|\hf(\{u,w\},T)|+|\hf(\{v,w\},T)|= |\hf|-|\hf\setminus \hf_{uvw}| \geq |\hf| -4\binom{n-3}{k-3}.
\]
By Claim \ref{claim-1} and $n> 42k$,
\begin{align}\label{ineq-key3}
|\hf(\{v,w\},T)|> \frac{1}{3}|\hf|-3\binom{n-3}{k-3}\geq \frac{8}{9}\binom{n-3}{k-2}-3\binom{n-3}{k-3}> \frac{7}{9}\binom{n-3}{k-2}.
\end{align}
Thus \eqref{ineq-key5} follows.
\end{proof}

Set
\begin{align*}
&|\hf(\{u,v\},T)|=\binom{n-3}{k-2}-f_{uv},\\[5pt] &|\hf(\{u,w\},T)|=\binom{n-3}{k-2}-f_{uw},\\[5pt]  &|\hf(\{v,w\},T)|=\binom{n-3}{k-2}-f_{vw}
\end{align*}
and set also
\begin{align*}
&|\hf(\{u\},T)|=g_u,\ |\hf(\{v\},T)|=g_v,\  |\hf(\{w\},T)|=g_w\ \mbox{and } \ |\hf(\emptyset,T)|=h,\ |\hf(T,T)|=m.
\end{align*}
Then it is easy to see that
\begin{align*}
&|\hf(u)|=2\binom{n-3}{k-2}-f_{uv}-f_{uw}+g_u+m,\\[5pt]
&|\hf(v)|=2\binom{n-3}{k-2}-f_{uv}-f_{vw}+g_v+m,\\[5pt]
&|\hf(w)|=2\binom{n-3}{k-2}-f_{uw}-f_{vw}+g_w+m,\\[5pt]
&|\hf|= 3\binom{n-3}{k-2}-f_{uv}-f_{uw}-f_{vw}+g_u+g_v+g_w+m+h.
\end{align*}
Note that by $C\geq 1$ and $3-2C=2\varepsilon$ we have
\begin{align*}
&\quad\ 3(|\hf|-C\Delta(\hf)) \\[5pt]
&\leq 3|\hf|-C(|\hf(u)|+|\hf(v)|+|\hf(w)|)\\[5pt]
&=3(3-2C)\binom{n-3}{k-2}-(3-2C)(f_{uv}+f_{uw}+f_{vw})+(3-C)(g_u+g_v+g_w)+(3-3C)m+3h\\[5pt]
&\leq 3(3-2C)\binom{n-3}{k-2}-(3-2C)(f_{uv}+f_{uw}+f_{vw})+(3-C)(g_u+g_v+g_w)+3h\\[5pt]
&\leq  3(3-2C)\binom{n-3}{k-2}-2\varepsilon(f_{uv}+f_{uw}+f_{vw})+2(g_u+g_v+g_w)+3h.
\end{align*}
By \eqref{ineq-key5},
\[
|\hf(\{u,v\},T)| > \frac{7}{9}\binom{n-3}{k-2} > 3\binom{n-4}{k-3}\geq \binom{n-4}{k-3}+\binom{n-5}{k-3}+\binom{n-6}{k-3}.
\]
Since $\hf(\{u,v\},T)$, $\hf(\{w\},T)$ are cross-intersecting and $\hf(\{u,v\},T)$, $\hf(\emptyset,T)$ are cross-intersecting,
 applying Lemma \ref{lem-key} with $d=3$ we obtain that
\[
\binom{n-3}{k-2}-f_{uv}+\frac{\binom{n-6}{k-2}}{\binom{n-6}{k-4}}g_w \leq \binom{n-3}{k-2}
\]
and
\[
\binom{n-3}{k-2}-f_{uv}+\frac{\binom{n-6}{k-2}}{\binom{n-6}{k-3}}h \leq \binom{n-3}{k-2}.
\]
Since $n\geq \frac{21k}{\varepsilon}$ implies that
\[
\frac{\binom{n-6}{k-2}}{\binom{n-6}{k-4}} =\frac{(n-k-2)(n-k-3)}{(k-2)(k-3)} > \frac{2}{\varepsilon}, \ \frac{\binom{n-6}{k-2}}{\binom{n-6}{k-3}} =\frac{n-k-3}{k-2} > \frac{1}{\varepsilon},
\]
it follows that
\[
g_w < \frac{\varepsilon}{2}f_{uv}\mbox{ and }  h <  \varepsilon f_{uv}.
\]
Similarly,
\[
g_v <\frac{\varepsilon}{2} f_{uw},\  h <   \varepsilon f_{uw},\ g_u < \frac{\varepsilon}{2} f_{vw}, \ h <   \varepsilon f_{vw}.
\]
Adding these inequalities,
\[
2(g_u+g_v+g_w)+3h <2\varepsilon(f_{uv}+f_{uw}+f_{vw}).
\]
Thus,
\[
3\gamma_C(\hf)=3(|\hf|-C\Delta(\hf))<3(3-2C)\binom{n-3}{k-2},
\]
the finial contradiction.
 \end{proof}

\section{A stability result for the diversity theorem}

Let us prove a lemma.

\begin{lem}\label{lem-key3}
Let $\hf\subset \binom{[n]}{k}$ be an intersecting family. If $\gamma(\hf)=\left(1-\alpha\right)\binom{n-3}{k-2}$ with $0\leq \alpha<1$ and  $n\geq \frac{36k}{1-\alpha}$, then there exists $T:=\{u,v,w\}\subset [n]$ such that
\begin{align}\label{ineq4-key0}
 |\hf(\emptyset,T)|\leq \binom{n-7}{k-4}\mbox{ and }|\hf(\{x\},T)|\leq \binom{n-4}{k-3},\ x=u,v,w
\end{align}
and for $P\in \{\{u,v\}, \{u,w\}, \{v,w\}\}$,
\begin{align}\label{ineq4-key5}
|\hf(P,T)|  \geq  \left(1-\alpha\right)\binom{n-3}{k-2}- |\hf(\emptyset,T)| -\sum_{x\in P}|\hf(\{x\},T)|.
\end{align}
\end{lem}

\begin{proof}
Let $|\hf(u)|=\Delta(\hf)$. Since $n\geq \frac{36k}{1-\alpha}\geq 36k$,  we have
\[
|\hf(\bar{u})|=\gamma(\hf)=\left(1-\alpha\right)\binom{n-3}{k-2}
=\left(1-\alpha\right)\frac{n-3}{k-2}\binom{n-4}{k-3}>36\binom{n-4}{k-3}.
\]
Let $\varepsilon =\frac{1-\alpha}{36}< \frac{1}{24}$.
Since $n\geq \frac{36k}{1-\alpha}=\frac{k}{\varepsilon}$, by Theorem \ref{thm-fw-3}  there exists  $v\in [n]\setminus \{u\}$ such that
\begin{align*}
|\hf(\bar{u},v)|> \left(\frac{2}{3}-\varepsilon\right)|\hf(\bar{u})|& \geq \left(\frac{2}{3}-\frac{1-\alpha}{36}\right)\left(1-\alpha\right)\binom{n-3}{k-2}\\[3pt]
&=\left(\frac{2}{3}-\frac{1-\alpha}{36}\right)\left(1-\alpha\right)\frac{n-3}{k-2}\binom{n-4}{k-3}\\[3pt]
&\geq \left(\frac{2}{3}-\frac{1-\alpha}{36}\right)\left(1-\alpha\right)\frac{36}{1-\alpha}\binom{n-4}{k-3}\\[3pt]
&>5\binom{n-4}{k-3}.
\end{align*}
By Lemma \ref{lem-key2} there exists $w\in [n]\setminus \{u,v\}$ such that \eqref{ineq4-key0} holds.

We are left to show \eqref{ineq4-key5}. Let $T=\{u,v,w\}$.
For $x\in T$ and  $T\setminus \{x\}=\{y,z\}$,
\begin{align*}
(1-\alpha)\binom{n-3}{k-2}=\gamma(\hf)&\leq |\hf(\bar{x})|=|\hf(\{y,z\},T)|+|\hf(\{y\},T)|+|\hf(\{z\},T)|+|\hf(\emptyset,T)|.
\end{align*}
 Thus,
\[
|\hf(\{y,z\},T)| \geq (1-\alpha)\binom{n-3}{k-2}-|\hf(\emptyset,T)|-|\hf(\{y\},T)|-|\hf(\{z\},T)|
\]
 and \eqref{ineq4-key5}  follows.
\end{proof}

One can derive Theorem \ref{thm-main2} from Lemma \ref{lem-key3}.

\begin{proof}[Proof of Theorem \ref{thm-main2}]
Note that $n\geq \frac{dk}{1-\alpha}$ and $d\geq 36$ imply $n\geq \frac{36k}{1-\alpha}$. By Lemma \ref{lem-key3} we infer that there exists $T:=\{u,v,w\}$ such that \eqref{ineq4-key0} and \eqref{ineq4-key5} hold.

Set
\begin{align*}
&|\hf(\{u,v\},T)|=\binom{n-3}{k-2}-f_{uv},\\[3pt] &|\hf(\{u,w\},T)|=\binom{n-3}{k-2}-f_{uw},\\[3pt]  &|\hf(\{v,w\},T)|=\binom{n-3}{k-2}-f_{vw}
\end{align*}
and set also
\begin{align*}
&|\hf(\{u\},T)|=g_u,\ |\hf(\{v\},T)|=g_v,\  |\hf(\{w\},T)|=g_w\ \mbox{and } |\hf(\emptyset,T)|=h.
\end{align*}

By \eqref{ineq4-key5},
\begin{align*}
&\binom{n-3}{k-2}-f_{uv}\geq \left(1-\alpha\right)\binom{n-3}{k-2}-h -g_u-g_v.
\end{align*}
It follows that
\begin{align*}
f_{uv}\leq \alpha\binom{n-3}{k-2}+h+g_u+g_v.
\end{align*}
Similarly,
\begin{align*}
f_{uw}\leq \alpha\binom{n-3}{k-2}+h+g_u+g_w,\
f_{vw}\leq \alpha\binom{n-3}{k-2}+h+g_v+g_w.
\end{align*}
Adding these inequalities, we get
\begin{align}\label{ineq4-3}
f_{uv}+f_{uw}+f_{vw}\leq 3\alpha\binom{n-3}{k-2}+3h+2(g_u+g_v+g_w).
\end{align}
By \eqref{ineq4-key0}, \eqref{ineq4-key5} and $n\geq\frac{dk}{1-\alpha}$, we have
\begin{align*}
|\hf(\{u,v\},T)|&\geq (1-\alpha)\binom{n-3}{k-2}-h-g_u-g_v\\[3pt]
 &> (1-\alpha)\frac{n-3}{k-2}\binom{n-4}{k-3}-3\binom{n-4}{k-3}\\[3pt]
 &> (d-3)\binom{n-4}{k-3}.
\end{align*}
Since $\hf(\{u,v\},T)$, $\hf(\{w\},T)$ are cross-intersecting and $\hf(\{u,v\},T)$, $\hf(\emptyset,T)$ are cross-intersecting, by Lemma
\ref{lem-key}  we have
\[
\binom{n-3}{k-2} -f_{uv} +\frac{\binom{n-d}{k-2}}{\binom{n-d}{k-d+2}}g_w \leq \binom{n-3}{k-2}
 \]
 and
 \[
\binom{n-3}{k-2} -f_{uv} +\frac{\binom{n-d}{k-2}}{\binom{n-d}{k-d+3}}h \leq \binom{n-3}{k-2}.
 \]
 It follows that
 \begin{align}\label{ineq4-4}
 g_w \leq \frac{\binom{n-d}{k-d+2}}{\binom{n-d}{k-2}}f_{uv},\ h \leq \frac{\binom{n-d}{k-d+3}}{\binom{n-d}{k-2}}f_{uv}.
 \end{align}
 Similarly,
 \begin{align}\label{ineq4-5}
g_v \leq \frac{\binom{n-d}{k-d+2}}{\binom{n-d}{k-2}}f_{uw},\ h \leq \frac{\binom{n-d}{k-d+3}}{\binom{n-d}{k-2}}f_{uw},\ g_u \leq \frac{\binom{n-d}{k-d+2}}{\binom{n-d}{k-2}}f_{vw},\ h \leq \frac{\binom{n-d}{k-d+3}}{\binom{n-d}{k-2}}f_{vw}.
 \end{align}
 Then
 \[
 3h+2(g_u+g_v+g_w) \leq \beta(f_{uv}+f_{uw}+f_{vw}),
 \]
where $\beta:= \frac{\binom{n-d}{k-d+3}+2\binom{n-d}{k-d+2}}{\binom{n-d}{k-2}}$. Note that $n\geq 36k$ implies
 \begin{align*}
 \beta<\frac{3\binom{n-d}{k-d+3}}{\binom{n-d}{k-2}}
 =\frac{3(k-2)(k-3)\ldots(k-d+3)}{(n-k-3)(n-k-4)\ldots(n-k-d+3)}
 <\frac{3}{35^2}<\frac{1}{6}.
 \end{align*}
 By \eqref{ineq4-3} we obtain that
 \[
 f_{uv}+f_{uw}+f_{vw} \leq \frac{3\alpha}{1-\beta} \binom{n-3}{k-2}.
 \]
 Using $\frac{1}{1-x} <1+2x$ for $x<\frac{1}{2}$,
 \[
  f_{uv}+f_{uw}+f_{vw} \leq 3\alpha(1+2\beta) \binom{n-3}{k-2} =3\alpha \binom{n-3}{k-2}+6\alpha\beta \binom{n-3}{k-2}< 4\alpha \binom{n-3}{k-2}.
\]
By \eqref{ineq4-1.2}, $n\geq \frac{dk}{1-\alpha}\geq dk$ and $d\geq 36$,
 \begin{align*}
 \binom{n-d}{k-2}\geq \frac{n-3-(d-3)(k-2)}{n-3}\binom{n-3}{k-2}\geq \frac{n-(d-3)k}{n}\binom{n-3}{k-2}\geq \frac{3}{d}\binom{n-3}{k-2}.
 \end{align*}
By \eqref{ineq4-4} and \eqref{ineq4-5},
\begin{align*}
|\hf\setminus \hf_{uvw}^*|=h+g_u+g_v+g_w&\leq\frac{\binom{n-d}{k-d+3}
+3\binom{n-d}{k-d+2}}{3\binom{n-d}{k-2}}(f_{uv}+f_{uw}+f_{vw})\\[3pt]
&\leq\frac{\binom{n-d+3}{k-d+3}
}{3\binom{n-d}{k-2}}\times 4\alpha \binom{n-3}{k-2}\\[3pt]
&\leq \frac{4}{3}\alpha \times \frac{d}{3}  \binom{n-d+3}{k-d+3}\\[3pt]
&\leq \frac{d\alpha}{2} \binom{n-d+3}{k-d+3}.
\end{align*}
By \eqref{ineq4-3},
\begin{align*}
|\hf_{uvw}\setminus \hf|= f_{uv}+f_{uw}+f_{vw} &\leq 3\alpha \binom{n-3}{k-2}+3(h+g_u+g_v+g_w)\\[5pt]
 &\leq 3\alpha \binom{n-3}{k-2}+\frac{3d\alpha}{2} \binom{n-d+3}{k-d+3}.
\end{align*}
\end{proof}

\section{Concluding remarks}

In the present paper, by using a maximum degree result in \cite{FW2022} we proved Magnan, Palmer and  Wood's $C$-diversity theorem for $1<C<\frac{3}{2}$ under a linear constraint of $n$.

Let $\hl=\{L_1,L_2,\ldots,L_7\}\subset \binom{[7]}{3}$ be the 3-graph formed by the seven lines of the Fano plane.
Let $\hl^c=\{[7]\setminus L\colon L \in \hl\}$ and let $\hl^+=\hl\cup \left(\binom{[7]}{4}\setminus \hl^c\right)$.
 For $n\geq 7$ and $k\geq 3$, let
\[
\hf_{\hl} =\left\{F\in \binom{[n]}{k}\colon F \cap [7]\in \hl\right\} \mbox{ and } \hf_{\hl^+} =\left\{F\in \binom{[n]}{k}\colon F \cap [7]\in \hl^+\right\}.
\]
It is easy to see that $|\hf_{\hl}| = 7\binom{n-7}{k-3}$, $|\hf_{\hl^+}| =7\binom{n-7}{k-3}+28\binom{n-7}{k-4}$,  $\Delta(\hf_{\hl})=3 \binom{n-7}{k-3}$ and $\Delta(\hf_{\hl^+})=3 \binom{n-7}{k-3}+16\binom{n-7}{k-4}$. Thus,
\[
\gamma_C(\hf_{\hl})=|\hf_{\hl}| -C\Delta(\hf_{\hl}) = (7-3C)\binom{n-7}{k-3}
 \]
 and
 \[\gamma_C(\hf_{\hl^+})=|\hf_{\hl^+}| -C\Delta(\hf_{\hl^+}) = (7-3C)\binom{n-7}{k-3}+(28-16C)\binom{n-7}{k-4}.
\]
For $C<\frac{3}{2}$ and $n\leq \frac{2(k-2)}{3-2C}$, we have
\begin{align*}
\gamma_C(\hf_{\hl}) = (7-3C)\binom{n-7}{k-3}>\frac{5}{2}\frac{k-2}{n-6} \binom{n-6}{k-2} >(3-2C)\frac{5}{4}\binom{n-6}{k-2}\geq (3-2C)\binom{n-3}{k-2}.
\end{align*}
Thus one needs at least  $n> \frac{2(k-2)}{3-2C}$ in Theorem \ref{thm-main}.

In \cite{MPW}, Magnan, Palmer and  Wood also determined the maximum $C$-diversity of an intersecting family  for $\frac{3}{2}\leq C\leq \frac{7}{3}$ and for $n$ sufficiently large with respect to $k$.

\begin{thm}[\cite{MPW}]\label{thm-mpw2}
Let $\hf\subset \binom{[n]}{k}$ be an intersecting family. For $n$ large enough,
\[
\gamma_C(\hf) \leq \left\{\begin{array}{ll}
                  (7-3C)\binom{n-7}{k-3}+(28-16C)\binom{n-7}{k-4}&\text{ for }\frac{3}{2}\leq C<\frac{7}{4}, \\[5pt]
                  (7-3C)\binom{n-7}{k-3}&\text{ for }\frac{7}{4}\leq C< \frac{7}{3},
                \end{array}\right.
\]
with equality holding if and only if $\hf$ is isomorphic to $\hf_{\hl^+}$ in the former case and $\hf_{\hl}$ in the latter (with $C\neq \frac{7}{4}$).
\end{thm}

It is natural to ask whether Theorem \ref{thm-mpw2} also holds under a linear constraint of $n$. However, our methods are not strong enough to achieve it.

Let us conclude this paper by a open problem concerning the ordinary diversity.

\begin{conj}
Suppose that $\hf\subset \binom{[n]}{k}$ is intersecting. If $n>4k$, then
\begin{align*}
\gamma(\hf) \leq \binom{n-3}{k-2}.
\end{align*}
\end{conj}

Theorem \ref{thm-fw-2} shows that the Conjecture holds for $n>36k$. However, it is open for the range $4k<n\leq 36k$.

\begin{appendix}
\section{A proof of Lemma \ref{thm-fk}.}

Two families $\ha$, $\hb$ are said to be cross $t$-intersecting if $|A\cap B|\geq t$ for all $A\in \ha$ and $B\in \hb$.

\begin{thm}[\cite{F78}]\label{thm-fk2}
Suppose that $\ha,\hb\subset \binom{[n]}{k}$ are cross $t$-intersecting, $|\ha|\leq |\hb|$. Then either $|\hb| \leq \binom{n}{k-t}$ or $|\ha| \leq \binom{n}{k-t-1}$.
\end{thm}

For a family $\hf\subset \binom{[n]}{k}$ and two integers $1\leq i<j\leq n$, one defines the $i\leftarrow j$ shift $S_{ij}$ by
\[
S_{ij}(\hf) =\{S_{ij}(F)\colon F\in \hf\},
\]
where
$$S_{ij}(F)=\left\{
                \begin{array}{ll}
                  F':=(F\setminus\{j\})\cup\{i\}, & j\in F, i\notin F \text{ and } F' \notin \hf; \\[5pt]
                  F, & \hbox{otherwise.}
                \end{array}
              \right.
$$
Obviously, $|S_{ij}(\hf)|=|\hf|$, $|S_{ij}(F)|=|F|$ and it is known (cf. e.g. \cite{F87}) that the $i\leftarrow j$ shift maintains the cross-intersecting property.

\begin{proof}[Proof of Lemma \ref{thm-fk}]
Suppose indirectly that \eqref{ineq-frankl2} does not hold. We claim that
\begin{align}\label{ineq2-1}
\max\{|\ha(j)|,|\hb(j)|\} <\binom{m-2}{\ell-2}+\binom{m-3}{\ell-2} \mbox{ for all } j\in [m].
\end{align}
Indeed, if $|\ha(j)|\geq \binom{m-2}{\ell-2}+\binom{m-3}{\ell-2}$ then by Lemma \ref{lem-key0} we have $|\hb(\bar{j})|\leq \binom{m-3}{\ell-2}$. It follows that
\[
|\hb(j)| =|\hb| -|\hb(\bar{j})| \geq 4\binom{m-2}{\ell-2}> \binom{m-2}{\ell-2}+\binom{m-3}{\ell-2}.
\]
Applying Lemma \ref{lem-key0} to $\hb(j)$ and $\ha(\bar{j})$ yields $|\ha(\bar{j})|\leq \binom{m-3}{\ell-2}$. This would prove \eqref{ineq-frankl2} contradicting the indirect assumption. Thus \eqref{ineq2-1} is proved.

Let us try and shift simultaneously $\ha$ and $\hb$. We claim that we never arrive at families  $\widetilde{A}$, $\widetilde{B}$ (at the first time) failing \eqref{ineq2-1}. By symmetry assume that after applying the shift $S_{ij}$,
\begin{align}\label{ineq2-2}
|\widetilde{\ha}(i)|\geq \binom{m-2}{\ell-2}+\binom{m-3}{\ell-2}.
\end{align}
By Lemma \ref{lem-key0} we have $|\widetilde{\hb}(\bar{i})|\leq \binom{m-3}{\ell-2}$.
Note that \eqref{ineq2-1} implies that
\[
|\widetilde{\hb}(i)| \leq |\hb(i)|+|\hb(j)| <4\binom{m-2}{\ell-2}.
\]
Hence,
\[
|\widetilde{\hb}|=|\widetilde{\hb}(i)|+|\widetilde{\hb}(\bar{i})|<5\binom{m-2}{\ell-2},
\]
a contradiction.

Now we proved that one can shift $\ha$, $\hb$ simultaneously without violating \eqref{ineq2-1}. By abuse of notation let $\ha$ and $\hb$ denote the shifted families that we eventually obtain. Then $\ha(\bar{1}), \hb(\bar{1})\subset \binom{[2,m]}{\ell}$ are cross 2-intersecting. By Theorem \ref{thm-fk2} and symmetry we may assume that $|\ha(\bar{1})|<\binom{m-1}{\ell-2}<2\binom{m-2}{\ell-2}$. This leads to $|\ha(1)|=|\ha|-|\ha(\bar{1})|>2\binom{m-2}{\ell-2}$, the final contradiction.
\end{proof}

\end{appendix}

\end{document}